\documentclass[11pt,letterpaper]{amsart}
\usepackage{amscd}
\usepackage{amssymb}
\usepackage{amsthm}
\usepackage{enumerate}
\usepackage[normalem]{ulem}
\usepackage{mathtools}
\usepackage[usenames,dvipsnames]{xcolor}
\usepackage{stmaryrd}
\usepackage[left=1in,top=1in,right=1in]{geometry}
\usepackage[mathscr]{eucal}
 \usepackage{cite}
\usepackage{upgreek}
\usepackage[bookmarks=false]{hyperref}
\usepackage[T2A]{fontenc}
\usepackage[utf8]{inputenc}
\usepackage{tikz-cd}
\usepackage{mathrsfs}

\newtheorem{introthm}{Theorem}

\newtheorem{thm}{Theorem}[section]
\newtheorem{lem}[thm]{Lemma}
\newtheorem{prop}[thm]{Proposition}
\newtheorem{cor}[thm]{Corollary}

\theoremstyle{definition}
\newtheorem{defn}[thm]{Definition}

\theoremstyle{remark}
\newtheorem{rem}[thm]{Remark}

\numberwithin{equation}{section}

\newcommand{\bB}{{\mathbb B}}
\newcommand{\bC}{{\mathbb C}}

\newcommand{\bF}{{\mathbb F}}

\newcommand{\bK}{{\mathbb K}}

\newcommand{\bN}{{\mathbb N}}

\newcommand{\bR}{{\mathbb R}}

\newcommand{\bX}{{\mathbb X}}

\newcommand{\bZ}{{\mathbb Z}}

\newcommand{\cH}{{\mathcal H}}

\newcommand{\cK}{{\mathcal K}}
\newcommand{\cL}{{\mathcal L}}

\newcommand{\cU}{{\mathcal U}}

\DeclareMathOperator{\id}{id}

\DeclareMathOperator{\sgn}{sgn}

\newcommand{\ip}[1]{\langle #1 \rangle}

\newcommand{\ee}{\varepsilon}

\begin{document}

\title{Strict comparison in reduced group $C^*$-algebras}

\author[T.Amrutam]{Tattwamasi Amrutam}
\address{\parbox{\linewidth}{Institute of Mathematics of the Polish Academy of Sciences, Ul. S'niadeckich 8,\\ 00-656 Warszawa, Poland}}
\email{tattwamasiamrutam@gmail.com}
\urladdr{https://tattwamasiamrutam.com}

\author[D. Gao]{David Gao}
\address{\parbox{\linewidth}{Department of Mathematics, University of California, San Diego, \\
9500 Gilman Drive \# 0112, La Jolla, CA 92093}}
\email{weg002@ucsd.edu}
\urladdr{https://sites.google.com/ucsd.edu/david-gao}

\author[S. Kunnawalkam Elayavalli]{Srivatsav Kunnawalkam Elayavalli}
\address{\parbox{\linewidth}{Department of Mathematics, University of California, San Diego, \\
9500 Gilman Drive \# 0112, La Jolla, CA 92093}}
\email{skunnawalkamelayaval@ucsd.edu}
\urladdr{https://sites.google.com/view/srivatsavke}

\author[G. Patchell]{Gregory Patchell}
\address{\parbox{\linewidth}{Department of Mathematics, University of California, San Diego, \\
9500 Gilman Drive \# 0112, La Jolla, CA 92093}}
\email{gpatchel@ucsd.edu}
\urladdr{https://sites.google.com/view/gpatchel}

\begin{abstract}
We prove that for every $n\in \mathbb{N}$ such that $n\geq 2$, the reduced group $C^*$-algebras of the countable free groups  $C^*_r(\mathbb{F}_n)$ have strict comparison. Our method works in a general setting: for every finitely generated acylindrically hyperbolic group $G$ with trivial finite radical and the rapid decay property, we have $C^*_r(G)$ have strict comparison.  This work also has several applications in the theory of $C^*$-algebras including: resolving Leonel Robert's selflessness problem for $C^*_r(G)$; uniqueness of embeddings of the Jiang-Su algebra $\mathcal{Z}$ up to approximate unitary equivalence into $C^*_r(G)$; full computations of the Cuntz semigroup of $C^*_r(G)$ and future directions in the $C^*$-classification program.  
\end{abstract}

\maketitle%
\begin{center}
    \emph{In memory of  Dr. S. Balachander.}
\end{center}

\section{Introduction}

\subsection{Strict comparison}
The subspaces of a finite dimensional Hilbert space $\cH$ can be compared and classified by their dimension; algebraically, the dimension of a subspace $\cK\subset \cH$ is the trace of the orthogonal projection onto $\cK$. For matrices $x,y$ one has the following natural notion of \emph{comparison}, that is $x\precsim y$ if $x = rys$ for some matrices $r, s$. It turns out here that the \emph{rank} is what precisely implements this comparison, that is  $\mathrm{rank}(x) \leq \mathrm{rank}(y)$ if and only if $x\precsim y$. Murray and von Neumann studied this comparison quite effectively in a broad setting of bounded operators on Hilbert spaces, in what are today called von Neumann algebras. In a broad class of von Neumann algebras known as {II$_1$ factors, they were able to establish a \emph{comparison theorem}, showing that the lattice of projections under the above subequivalence is determined exactly by the unique finite trace \cite{Murray-Neumann36}. Notably, in this context the trace takes on continuous values, that is, for all $t\in [0,1]$, there is a projection $p\in \mathcal{P}(M)$ with $\tau(p)=t$. 

Such a comparison property allows for an algebraic perspective that is useful in operator algebras. Naturally one would want an analogue of the above Murray von Neumann comparison theorem in $C^*$-algebras.  However, the behavior of $C^*$-algebras is very different as evidenced, for instance, by the typical lack of projections. Hence, an analogue of this property is not only subtle to define for all unital simple $C^*$-algebras but also will not hold automatically in all situations like in the setting of II$_1$ factors. The appropriate analogue here is called \emph{strict comparison}, initially considered by Blackadar \cite{Blackadar_1989}. This can be rephrased in a much more algebraic manner in terms of a certain notion of order-completeness (\emph{almost unperforation}) of the important invariant called the Cuntz semigroup (see \cite{Rordam}, and for instance, \cite{gardella2022moderntheorycuntzsemigroups} or paragraph 3.4 of \cite{antoine2024purecalgebras} for more details). 

Let $A$ be a $C^*$-algebra and let $\mathbb{K}$ denote the compact operators on $\ell^2$. If $a,b \in (A \otimes \bK)_+$ are positive elements, we say that $a \precsim b$ (Cuntz subequivalent) if there are $v_n \in A \otimes \bK$ such that $v_n^*bv_n \to a$. Now given $\tau$ a trace on $A$, there is a canonical extension of $\tau$ to $(A \otimes \bK)_+$ (denoted by $\tau$ again). The dimension function here with respect to $\tau$ is defined on $(A\otimes \bK)_+$ by
$d_{\tau}(a):=\lim_n \tau(a^{\frac{1}{n}})$  (note that for positive operators in matrices, $d_\tau$ precisely gives the rank of the operator).
If $a,b\in (A \otimes \bK)_+$ satisfy $a \precsim b$, then $d_{\tau}(a)\le d_{\tau}(b)$. A unital, simple, exact $C^*$-algebra $A$ with unique trace $\tau$ has {strict comparison} if whenever $a,b \in (A \otimes \bK)_+$ are such that $d_\tau(a) < d_\tau(b)$, then $a \precsim b$. Note that many of the $C^*$-algebras considered in the results of our paper will be unital, simple, exact and with unique trace (this includes the reduced $C^*$-algebras of hyperbolic groups \cite{ADAMS1994765}), so that the above definition will suffice. In full generality, the definition will involve the space of quasitraces and is more subtle. 

Despite being such a fundamental and natural property of a $C^*$-algebra, proving strict comparison is difficult and has far-reaching applications (please see subsection \ref{applications section} where several applications arising from our present work are described). Strict comparison occupies an important place as a hypothesis in  Elliott's classification program, which seeks to classify simple, nuclear $C^*$-algebras by $K$-theoretical and tracial data.  This program has seen immense success in recent times and has resulted in a substantial body of work due to several mathematicians (see for instance \cite{GonglinNiu1, GongLinNiu2, uniformgamma, carrión2023classifyinghomomorphismsiunital, winter2017structurenuclearcalgebrasquasidiagonality, white2023abstractclassificationtheoremsamenable}). Note also the influential Toms-Winter conjecture predicts that three regularity properties of very different natures, namely topological (finite nuclear dimension), functional analytic ($\mathcal{Z}$-stability), and algebraic (strict comparison), are equivalent for a separable, unital, simple, {non-elementary}, nuclear $C^*$-algebra $A$. These three seemingly independent regularity properties, which are
conjectured to be equivalent, are exactly the conditions under which one considers a
$C^*$-algebra to belong to the classification program.

\subsection{Context and main results} A prototypical example of a  reduced group $C^*$-algebra is the one arising from {the} countable free group $\mathbb{F}_2$. The study of this $C^*$-algebra has been extremely important for opening up several deep insights in the subject. Some examples of this include the following fundamental results: unique trace and simplicity \cite{Powers} (see also the major works \cite{kennedy, haagerup2016newlookcsimplicityunique}  which were inspired by Powers' insights), lacking non-trivial projections \cite{projectionsPicu}, $K$-theory computations \cite{projectionsPicu}, showing stable rank 1 \cite{stablerankHaagerup} in the sense of Rieffel \cite{rieffelmain}, and showing \emph{strong convergence} matrix models \cite{HTExt} (see also \cite{louder2022strongly} whose ideas inspired the present paper).   {D}espite these important developments, the fundamental question of proving strict comparison for $C^*_r(\mathbb{F}_2)$ has evaded experts.   Within the community, it has been a widely-known open problem with initial motivation stemming from the work of Dykema-R{\o}rdam in 1998 (\cite{dykema1998projections, dykema2}, see also \cite[Proposition 6.3.2]{robert2012classification}), which addresses comparison in $C^*_r(\mathbb{F}_\infty)$. Note that the ideas behind the proof for strict comparison for $C^*_r(\mathbb{F}_\infty)$ \emph{cannot} be readily used to prove this result for  $C^*_r(\mathbb{F}_2)$. For instance, the proof there uses the ability to find freely independent elements from any finite subsets, up to norm perturbations, which holds inside $C^*_r(\mathbb{F}_\infty)$. However such behavior is not clear at all to witness inside  $C^*_r(\mathbb{F}_2)$. {Strict comparison for $C^*_r(\bF_2)$} has been explicitly stated as an open question in the work of Robert in 2012 \cite{robert2012classification}  and in the recent substantial problem list in $C^*$-algebras \cite{questions}, see Problems 86 and 88. 

Robert in \cite{robert2023selfless} asks a very interesting, more general question of selflessness for these $C^*$-algebras: the existence of a Haar unitary that is freely independent (in the reduced sense) from $C^*_r(G)$ viewed as the diagonal inclusion in its {$C^*$}-ultrapower. Interestingly, this property is well-studied and is of interest in the II$_1$ factor context where it turns out that such freely independent unitaries can always be constructed, due to beautiful work of Popa from  \cite{popa1995free}. Robert obtained a crucial insight that this selflessness property in $C^*$-algebras implies strict comparison through a soft ultrapower argument. For natural reasons, this is of independent interest from a continuous model theory point of view.

Another key motivation for showing strict comparison is that it is crucial in fully
computing the Cuntz semigroup, introduced {by Cuntz}
in his fundamental investigation on the existence of traces on
$C^*$-algebras \cite{Cuntzdimension}. Over several decades, Cuntz semigroup theory has evolved into a very mature field of research \cite{gardella2022moderntheorycuntzsemigroups}. It has been a long-sought  problem to obtain full computations of the Cuntz semigroup of reduced group $C^*$-algebras in a general sense and has been described in \cite{gardella2022moderntheorycuntzsemigroups} as ``may be out of reach''. The problem of computing the Cuntz semigroup for $C^*_r(\mathbb{F}_2)$ is even explicitly stated as Problem 16.4 therein. All of the problems described above are settled in this paper.

\begin{introthm}\label{main1}
    For any $n\in \mathbb{N}$ such that $n\geq 2$, $C^*_r(\mathbb{F}_n)$ has strict comparison. In fact, $C^*_r(\mathbb{F}_n)$  is selfless in the sense of Robert \cite{robert2023selfless}.
\end{introthm}

  Thanks to major developments   \cite{Rordam, MatuiSato}, we now know that for unital, separable, simple, unique trace, nuclear, infinite-dimensional $C^*$-algebras, strict comparison and $\mathcal{Z}$-stability are equivalent. Importantly, at the moment, outside the  $\mathcal{Z}$-stable case there are only a few examples of strict comparison arising from groups, such as infinite reduced free products like $C^*_r(\mathbb{F}_{\infty})$ (\cite[Proposition 6.3.2]{robert2012classification}). With our next main result, we greatly expand the list of known examples. To motivate this, we need to discuss the following notion.  We say a group $G$ has the \emph{rapid decay property} if, for every finite generating set $X$, there exists a polynomial $P$ such that for every element $a\in \mathbb{C}G$ supported on $B_X(r)$,
$\left\|\lambda_G(a)\right\|\le P(r)\|a\|_2.$ The rapid decay property (see the survey \cite{sapir2015rapid}) was first proved by Haagerup for the free groups \cite{HaagerupGod}. It was then generalized to Gromov hyperbolic groups \cite{Jolissaint, delaharpe}, relatively hyperbolic groups with peripheral subgroups having rapid decay \cite{sapir2015rapid}, cocompact lattices in $SL_3(\mathbb{R})$ or $SL_3(\mathbb{C})$  \cite{lafforgue, Chatterji}, mapping class groups of surfaces \cite{behrstockcentroid}, large type Artin groups \cite{holt}, heirarchically hyperbolic groups \cite{heir1}, and more. Rapid decay is extremely influential and has been used to prove several important problems, including the Novikov conjecture
\cite{CONNES1990345}, {the} Baum-Connes conjecture \cite{LafforgueBaumConnes}, and others including \cite{Grigorchuk}. Before we state our result, we will need to recall one more important and influential notion from geometric group theory \cite{osinicm}. A group is said to be \emph{acylindrically hyperbolic} if it admits a non-elementary acylindrical action on a Gromov hyperbolic metric space. We direct the reader to \cite{osin2016acylindrically} for context, examples, and references. Recall also the finite radical for a group is the largest finite normal subgroup. Within the class of acylindrically hyperbolic groups $G$, it turns out that the finite radical being trivial is equivalent to $C^*_r(G)$ having unique trace, which is in turn equivalent to $C^*_r(G)$ being simple (see Theorem 1.4 of \cite{Breuillard_2017} and Theorem 2.35 of \cite{dahmani2017hyperbolically}). We are now ready to state our next result:

\begin{introthm}\label{main2}

Let $G$ be an arbitrary finitely generated countable group that is acylindrically hyperbolic with trivial finite radical and, in addition, enjoys the rapid decay property. Then $C^*_r(G)$ has strict comparison. In fact, we have $C^*_r(G)$ is selfless in the sense of Robert \cite{robert2023selfless}.

\end{introthm}

 Note that the trivial finite radical assumption in the above theorem is necessary due to the fact that these $C^*$-algebras will have unique trace as a consequence of being selfless (see also \cite{Breuillard_2017}). The following are some examples of groups (with an added trivial finite radical assumption) that fit into the above result (see \cite{osin2016acylindrically, dahmani2017hyperbolically}): non-elementary hyperbolic groups; non-elementary relatively hyperbolic groups relative to peripheral subgroups with the rapid decay property, {which includes} free products of groups with {the} rapid decay property; mapping class groups of compact orientable surfaces \cite{bestvina}; graph products of rapid decay groups over graphs without a splitting \cite{minasyan}; heirarchically hyperbolic groups \cite{heir2}. In light of \cite{lafforgue}, it is very timely and interesting to ask if our main results hold in the setting of higher rank lattices, especially cocompact lattices in $SL_3(\mathbb{R})$ \footnote{In a recent development \cite{vigdorovich2025structuralpropertiesreducedcalgebras}, I. Vigdorovich has settled this problem.}. Our proof strategy has no bearing in this setting, and therefore we leave it as an important open question arising from our work. Notably, even stable rank $1$ seems to be an open problem for reduced group $C^*$-algebras of higher rank lattices. We also speculate that our proof technique might open up a way to prove {strict comparison} for even wider families of groups, such as all $C^*$-simple groups with rapid decay, see Remark \ref{simple}. \footnote{In a very recent development \cite{ozawa2025proximalityselflessnessgroupcalgebras}, N. Ozawa has developed a new approach that proves selflessness for the reduced $C^*$-algebras of groups admitting \emph{extremely proximal} boundaries.}

 Importantly, also see Remark \ref{rem:subsemigroup}, which describes why the rapid decay assumption cannot be removed with our current approach. However, we point out that our result gives new examples of groups which do \emph{not} have rapid decay property but whose reduced $C^*$-algebras are selfless. Indeed by Theorem \ref{main1} and Theorem 4.2 of \cite{robert2023selfless}, we have $C^*_r(\mathbb{F}_2*SL_3(\mathbb{Z}))$ is selfless, however  $\mathbb{F}_2*SL_3(\mathbb{Z})$ does not have rapid decay (see Remark 8.6 in \cite{ValetteBC}). 
 
\subsection{Applications to classification and Cuntz semigroup theory}\label{applications section}

 It is a very standard fact that every II$_1$ factor contains a copy of the hyperfinite II$_1$ factor $R$ uniquely up to pointwise approximate unitary equivalence.  The most naïve analogue of this in $C^*$-algebras would be that every simple, unital, infinite-dimensional $C^*$-algebra contains the Jiang-Su algebra $\mathcal{Z}$ uniquely up to pointwise approximate unitary equivalence.  Both existence and uniqueness problems for embeddings of $\mathcal{Z}$ tend to be difficult, see \cite[Section 18]{questions}.  It is known that there are simple, unital, infinite-dimensional $C^*$-algebras that do not contain $\mathcal{Z}$ \cite{dadarlat}.  It remains open if there is any unital $C^*$-algebra $A$ with two non-equivalent unital embedding of $\mathcal{Z}$. Our work yields the following concrete application. The result below follows from Proposition 6.3.1 in \cite{robert2012classification} (which can be applied since the reduced $C^*$-algebras of the groups in consideration have unique quasitrace from Haagerup's famous quasitrace theorem \cite{Haageruptraces}, and have stable rank one from \cite{Osinacylindrical} see also \cite{raum2024twistedgroupcalgebrasacylindrically}) and our Theorem \ref{main2}.   Note also that the following result addresses problems 62 and 64 in \cite{questions}:

\begin{introthm}
        For any exact group $G$ in the class of groups covered in Theorem \ref{main2}, there is a unital embedding of the Jiang-Su algebra $\mathcal{Z}\to C^*_r(G)$, and this embedding is unique up to pointwise approximate unitary equivalence.
\end{introthm}

As explained in the beginning of the introduction,  strict comparison is a crucial hypothesis that is used in classifying maps up to $K$-theoretic and tracial data in the nuclear {setting}. In the natural progression of the current landscape of classification, we expect new developments to emerge in the more general setting of non-nuclear $C^*$-algebras, wherein strict comparison will still naturally play a crucial {role} as a hypothesis for the co-domain for classifying maps between $C^*$-algebras up to the total invariant. Our results here will provide a wealth of interesting new and natural co-domains for this type of classification phenomena.

In \cite{Cuntzsemigroup}, Coward-Elliott-Ivanescu gave a new picture of
the Cuntz semigroup in terms of Hilbert modules; they {showed} that for
$C^*$-algebras of stable rank one, the Cuntz semigroup encodes precisely the
isomorphism classes of countably generated Hilbert modules, similarly to
how the Murray-von Neumann semigroup encodes isomorphism classes of
finitely generated projective modules. The theory of Cuntz semigroups is
particularly well-developed for $C^*$-algebras of stable rank one \cite{ThielDuke}.       An important notion in this context is pureness, which was introduced by
Winter in his seminal study \cite{Winterinvent} of regularity properties for simple,
nuclear $C^*$-algebras. Pureness can be thought of as $\mathcal{Z}$-stability at the
level of the Cuntz semigroup, since in \cite{Tensorproductsthiel}, it was shown that a
$C^*$-algebra $A$ is pure if and only if its Cuntz semigroup $\mathrm{Cu}(A)$
tensorially absorbs $\mathrm{Cu}(\mathcal{Z})$. More recently, it was shown that pure
$C^*$-algebras form a robust class \cite{antoine2024purecalgebras}, and as observed in \cite[Example 5.14]{antoine2024purecalgebras}, if $G$ is an exact, infinite, $C^*$-simple group, then $C^*_r(G)$ has
strict comparison if and only if $C^*_r(G)$ is pure. If $A$ is a unital, separable, simple, stably finite, pure
$C^*$-algebra, then the Cuntz semigroup can be computed as
$\mathrm{Cu}(A) = V(A) \sqcup \mathrm{LAff}(QT(A))_{++}$ (see Theorem 2.6 in \cite{browntoms}). Here
where $QT(A)$ denotes the quasitracial states on $A$, $V(A)$ denotes the Murray-von Neumann semigroup and $\mathrm{LAff}(K)$ means the lower-semicontinuous, affine functions
$K$ to $(-\infty,\infty]$ if $K$ is a compact convex set, and $\mathrm{LAff}(K)_{++}$ means
the subset of strictly positive elements, that is, the
lower-semicontinuous, affine functions $K$ to $(0,\infty]$. If $K$ contains only one point, then $\mathrm{LAff}(K)_{++} = (0,\infty]$. If $A$ is exact, then
$QT(A)=T(A)$ by a famous result of Haagerup \cite{Haageruptraces}. Thus, if G is an exact, infinite, C*-simple group such that $C^*_r(G)$ has
strict comparison, then its Cuntz semigroup can be computed as
$\mathrm{Cu}(C^*_r(G)) = V(C^*_r(G)) \sqcup [0,\infty]$
since such $C^*$-algebras have a unique trace \cite{Breuillard_2017}.  Combined with our main results, this answers Problem 16.4 from \cite{gardella2022moderntheorycuntzsemigroups} which asks to fully compute the Cuntz semigroup for $C^*_r(\mathbb{F}_n)$, which is a Cuntz semigroup version of Blackadar’s fundamental comparison property \cite{Blackadar_1989} for projections in $C^*_r(\mathbb{F}_n)$.

\begin{introthm}
The Cuntz semigroup of $C^*_r(\mathbb{F}_2)$ is $\mathbb{N}\sqcup [0,\infty]$.
\end{introthm}

Note that, our Theorem \ref{main2} can be used to fully compute the Cuntz semigroup for the vast family of exact acylindrically hyperbolic groups with trivial finite radical and with rapid decay property as described above. This partially settles the general problem in the paragraph before Problem 16.4 in \cite{gardella2022moderntheorycuntzsemigroups}, obtaining a full computation for many classes of $C^*$-simple groups. For  examples of groups with stable rank one, see \cite{stablerankHaagerup, Osinacylindrical, raum2024twistedgroupcalgebrasacylindrically}. Note our work also recovers with a new proof, stable rank one for the families of groups in the statement of Theorem \ref{main2}  (via Theorem 3.1(2) of \cite{robert2023selfless}).

\subsection{Comments to the reader}
The inroad to prove strict comparison is to prove the more general free independence theorem in the norm ultrapower of the reduced group $C^*$-algebra, which is the \emph{selflessness} property from \cite{robert2023selfless}. As mentioned previously, our inspiration for the argument to achieve this came from a recent \emph{very} beautiful paper by Louder and Magee \cite{louder2022strongly}. Curiously, our original approach to solving the problem of strict comparison was (motivated in a sense by the work of the third author \cite{sri}) to use strong convergence methods in random matrix theory, and this is what indirectly led us to \cite{louder2022strongly}. Now, in order to prove this free independence theorem, one benefits by using the rapid decay property (see \cite{sapir2015rapid, HaagerupGod}), which allows one to control the operator norm of elements in the group ring in terms of the radius of the ball of support and the 2-norm. A key next step is a quantitative analysis of the eventual faithfulness of certain tailored retraction maps that one constructs using the group structure; for instance, free product structure, or, more generally, an acylindrical action on a hyperbolic space. This is used to obtain a particularly strong existential inclusion of the group into the free product of itself with a copy of $\mathbb{Z}$. Piecing these ideas together allows one to push this inclusion {to} the level of the reduced group $C^*$-algebra, as needed.

\subsection*{Concerning the organization} The proofs of these results split into a hands-on computational part (showing selflessness in our sense, see Definition \ref{selfless defn}) and a soft $C^*$-algebra part (using rapid decay \emph{à la} Louder-Magee, see Theorem \ref{rapiddecaytrick}), and then reaching the finish line for strict comparison with Robert's \cite{robert2023selfless}. If the reader would like to just focus on understanding the proof for arbitrary free products with rapid decay, we direct them to a self contained approach via the combination of Proposition \ref{F2-is-selfless} and Theorem \ref{rapiddecaytrick}. We also remark to the reader that after our paper was released publicly, M. Magee kindly shared with us his proof (via a quantitative Baumslag lemma) for just the non abelian free group case. This is in Subsection \ref{mageesubsection}.

On the other hand, the proof of Theorem \ref{main2} is quite involved, but the quantitative part is somewhat natural from a geometric point of view.  Appendix \ref{appendix main} contains a quantitative version of a key theorem of Wenyuan Yang that is needed to prove our result \ref{main2}. This is quite technical and we recommend the interested reader to side by side refer to \cite[Corollary 3.4]{YANG_2014}.

\subsection*{Acknowledgments} This result was proved during the visit of the first author to UCSD in Fall 2024, supported by the third author's NSF grant DMS 2350049 and the Institute of Mathematics of the Polish Academy of Sciences. The second and fourth author were supported in part by NSF grant DMS 2153805. We thank our colleagues J. Gabe, D. Jekel, N.C. Phillips, L. Robert, M. R{\o}rdam, C. Schafhauser, G. Szabo, H. Thiel and S. White for extremely helpful comments on the introduction. We sincerely thank all the experts we shared our earlier draft with, for providing us great encouragement and helpful comments to improve our paper. We also thank our colleagues J. Behrstock, F. Fournier-Facio, J. Huang, Y. Lodha, M. Magee, D. Osin, A. Sisto, and W. Yang for helpful pointers and feedback. We thank the anonymous referee for several very helpful and detailed comments that improved the paper.

\section{Notation}

Let $\Gamma$ be a countable discrete group. If $Y$ is a generating set for $\Gamma$ then we denote by $B_Y(R)$ the elements of the $R$-ball around the identity in the Cayley graph of $\Gamma$ with the generating set $Y$. Let $\bC\Gamma$ denote the complex group algebra associated to $\Gamma$. We equip $\bC\Gamma$ with three separate norms. Let $x = \sum{x_g}g$, then
\begin{align*}
    \|x\|_1 &:= \sum_{g\in \Gamma}|x_g|; \\
    \|x\|_2^2 &:= \sum_{g\in \Gamma}|x_g|^2;\\
    \|x\| &:= \sup\{\|xy\|_2 : y\in\bC\Gamma, \ \|y\|_2 \leq 1\}.
\end{align*}
We call these norms on $\bC\Gamma$ the \emph{$\ell^1$-norm}, the \emph{$\ell^2$-norm}, and the reduced \emph{operator norm} respectively. As suggested by the definition of the operator norm, each element of $\bC\Gamma$ can be viewed as a bounded operator on the Hilbert space $\ell^2(\Gamma)$, by the left regular representation $\lambda_\Gamma: \bC\Gamma\to \mathbb{B}(\ell^2\Gamma)$ (define $\lambda_\Gamma(g)(\delta_h)= \delta_{gh}$ and extend linearly). Note therefore we can write $\|x\|= \|\lambda_\Gamma(x)\|$ when viewed as a bounded operator on $\ell^2(\Gamma)$.  Note that for any element $x\in \mathbb{C}\Gamma$, one has the following elementary inequality $\|x\|_2\leq \|x||\leq \|x\|_1.$
We define the \emph{reduced group $C^*$-algebra} $C_r^*(\Gamma)$ to be the operator norm closure of $\bC\Gamma$ inside of $\bB(\ell^2(\Gamma))$. We will also need the following elementary fact, which follows from the triangle inequality: if $\phi:\Gamma\to G$ is a group homomorphism then $\phi$ extends to an $\ell^1$-contractive map between $\bC\Gamma$ and $\bC G$.

A \emph{state} on a unital $C^*$-algebra $A$ is a positive linear functional $\phi$ such that $\phi(1) = 1.$ We say that a state $\phi$ is \emph{faithful} if $\phi(x^*x) > 0$ for any nonzero $x\in A$. We call $\phi$ a \emph{trace} if $\phi(xy)=\phi(yx)$ for all $x,y\in A$. Let $A_1,A_2\subset (A,\phi)$ be subalgebras. We say that $A_1$ and $A_2$ are \emph{freely independent} relative to $\phi$ if whenever $a_k \in A_{i_k}$ with $\phi(a_k)=0$ and $i_k \neq i_{k+1}$ for $1\leq k< n$ then $\phi(a_1\cdots a_n) = 0.$ If $(B_1,\rho_1),(B_2,\rho_2)$ are unital $C^*$-algebras with faithful states, then there is a unital $C^*$-algebra with a faithful state $(B,\rho)$ and embeddings $\pi_i:(B_i,\rho_i)\hookrightarrow (B,\rho)$ such that $\pi_1(B_1)$ and $\pi_2(B_2)$ are freely independent relative to $\rho.$ We call $(B,\rho)$ the \emph{reduced free product} of $(B_1,\rho_1)$ and $(B_2,\rho_2)$; we omit the states when they are clear from context (see \cite{avitzour1982free,voiculescu1985symmetries}), and write $B_1*B_2$. We have in this context that whenever $G_1$ and $G_2$ are countable groups, $C^*_r(G_1*G_2)\cong C^*_r(G_1)*C^*_r(G_2)$, see Proposition 1.5.3 in \cite{VDN1992}.

Let $\cU$ be a free ultrafilter on the natural numbers. For a discrete group $\Gamma$ we denote by $\Gamma^\cU$ the following quotient: 
$$\Gamma^\cU := \left(\prod_{\bN}\Gamma\right) / N_\cU,$$
where $N_\cU$ consists of all sequences which are eventually identity with respect to $\cU$. More precisely, a sequence $(g_n)_n\in N_\mathcal{U}$ if and only if there is $S\in \cU$ such that $g_n = 1_\Gamma$ for all $n\in S.$ We call $\Gamma^\cU$ an \emph{ultrapower} of $\Gamma$. If $A$ is a $C^*$-algebra, we define $A^\cU$ as
$$A^\cU := \{(x_n)_n : x_n\in A, \ \sup_n\|x_n\|<\infty\} / I_\cU$$
where $I_\cU$ is the ideal consisting of all norm-bounded sequences $(x_n)$ in $A$ with the following property: for all $\ee>0$ there is $S\in \cU$ such that for all $n\in S$, $\|x_n\| < \ee$. We call $A^\cU$ an \emph{ultrapower} of $A$. Note that under the definition  $\|(x_n)_{\mathcal{U}}\|_{A^\mathcal{U}}= \lim_{\mathcal{U}}\|x_n\|$, and the natural product structure, $A^\cU$ is a $C^*$-algebra. 

\section{Proof of main results}

\subsection{A new property for groups} The following is a quantitative strengthening of the definition of the mixed identity free (MIF) property for groups \cite{MIF}, and our terminology is motivated by \cite{robert2023selfless} and is very much related to the quantitative Baumslag lemma from \cite{louder2022strongly}.

\begin{defn}\label{selfless defn}
     We say a group with a finite generating set $(G,X)$ is \emph{selfless}  if there is a function $f:\bN \to \bR_+$ with $\liminf_n f(n)^{\frac{1}{n}} = 1$ such that for all $n\geq 1$, there is an epimorphism $\phi_n : G * \ip{a} \to G$ such that $\phi_n|_G = \id_G$, $\phi_n$ is injective on $B_{X\cup\{a\}}(n)$, and $\phi_n(B_{X\cup\{a\}}(n)) \subset B_X(f(n))$. Here, $a$ is a torsion-free generator so that $G * \ip{a} \cong G * \bZ$.
\end{defn}

The following is a particularly motivating example:

\begin{prop}\label{F2-is-selfless}
    If $(G, X)$ contains a torsion-free element and $(H, Y)$ is an infinite group, then $G*H$ is selfless. In particular, $\bF_n$ is selfless for all $n\geq 2.$
\end{prop}

\begin{proof}
    Let $g\in G$ be torsion-free. Set $p = |g|_X$. For each $n$ let $h_n\in H$ be an element such that $|h_n|_Y = 2n+1$. Define $\phi_n : G*H*\ip{a} \to G*H$ by $\phi_n|_{G*H} = \id_{G*H}$ and $\phi_n(a) = h_ngh_n^{-1}$. It is clear that $\lim_n |\phi_n(a)|_{X\cup Y}^{\frac{1}{n}} = 1$, and therefore 
    $$\phi_n(B_{X \cup Y\cup\{a\}}(n)) \subset B_{X\cup Y}(f(n));$$
    for some function $f$ where $\lim_n f(n)^{\frac{1}{n}} = 1.$

    It remains to show that $\phi_n$ is injective on $B_{X\cup Y \cup\{a\}}(n)$. It suffices to show that, if
    $$s_1t_1\cdots s_{m-1}t_{m-1}s_m = e$$
    where $s_i \in G*H$ are such that $|s_i|_{X\cup Y} \le 2n$ for all $1 \le i \le m$ and $s_i\neq e$ for all $1 < i < m$, and $t_i = h_ng^{p_i}h_n^{-1}$ for some nonzero integers $p_i$, then $m=1$ and $s_1=e$.

    Assume that $m>1$. We expand out the expression $s_1t_1\cdots s_{m-1}t_{m-1}s_m$ and regroup the product as follows:
    $$(s_1 h_n)g^{p_1}(h_n^{-1}s_2h_n)g^{p_2}(h_n^{-1}s_3\cdots s_{m-1}h_n)g^{p_{m-1}}(h_n^{-1}s_m).$$

    We claim that each term in parentheses begins and ends with a letter in $H$. Indeed, for each $s_i\in H$ (necessarily nontrivial) this is immediate. Otherwise, $s_i$ begins either with a letter in $G$ or a letter in $H$ of length at most $2n$ followed by a letter in $G$; $s_i$ ends in a similar way. Therefore $h_n^{-1}s_ih_n$ begins and ends with a letter in $H$. Since each $p_i$ is nonzero, we now see that the above product is a nontrivial word in the free product; this finishes the proof.
\end{proof}

\subsection{A general family of selfless groups}

We will prove here that all finitely generated acylindrically hyperbolic groups  with trivial finite radical are selfless. For this subsection, for background notation and preliminaries on acylindrically hyperbolic groups, please refer to the beautiful memoir \cite{dahmani2017hyperbolically}, and also \cite{osin2016acylindrically}. At the outset we wish to express our gratitude to the following two experts for their guidance.
Firstly we sincerely thank D. Osin who first pointed out to us that our selflessness for free products (Proposition \ref{F2-is-selfless}) will hold in the general setting of acylindrically hyperbolic groups. 
Secondly we are very grateful to W. Yang for kindly sharing with us his ideas towards a different approach. First we describe the outline for our proof: 

\subsubsection*{Outline of proof}

The key step in the proof is to use a hyperbolically embedded $\mathbb{F}_2$ subgroup to construct the sequence of elements going to infinity, which will implement selflessness. The fact that this embedding is a Lipschitz quasi-retract \cite{dahmani2017hyperbolically} is crucial to relate word length inside $\bF_2$ and $G$. Then to check that these elements don't form a trivial word with a ball with a controlled co-growth, we have to use a quantitative version of  Wenyuan Yang's ``admissible path lemma'' in \cite{YANG_2014}, and argue that any word yields a quasi-geodesic. The proof crucially uses both acylindricity and geometry of hyperbolic space. Roughly, it says that a concatenation of quasigeodesics which travels alternatively near contracting subsets and leaves them in an orthogonal way, is a quasigeodesic. The careful quantitative analysis of this argument which is needed for our proof is done in Appendix \ref{appendix main}.

\begin{thm}\label{mainselfless}
    Let $G$ be a finitely generated acylindrically hyperbolic group with trivial finite radical. Then $G$ is selfless. 
\end{thm}
\begin{proof}
    
    Fix a finite generating set $S$ for $G$. Since $G$ is acylindrically hyperbolic and has trivial finite radical, by \cite[Theorem~6.14]{dahmani2017hyperbolically}, there is a hyperbolically embedded copy of $\bF_2$ in $G$. Denote by $a,b$ the generators of $\bF_2$. Without loss of generality, assume $a,b\in S$. By the proof of \cite[Theorem~4.31]{dahmani2017hyperbolically}, there is a constant $K>0$ and a finite generating set $Y$ of $\bF_2$ such that for all $h\in \bF_2$, $|h|_Y \leq |h|_S$ and $|h|_S \leq K|h|_Y$. Since $\{a,b\}$ and $Y$ are both finite generating sets for $\bF_2$, we have that $|h|_{a,b}$ and $|h|_{Y}$ are quasi-equivalent. Therefore there are constants $C,D$ such that whenever $h\in\bF_2,$ $|h|_{\{a,b\}}\leq C|h|_S$ and $|h|_S \leq D|h|_{\{a,b\}}$.

    Set $g_n = b^nab^{-n}$. We note that each $g_n$ is primitive. Note that for all $n,$ $\bF_2 = \ip{b}*\ip{g_n}$. By hyperbolicity of free factors (\cite[Example~2.2(c)]{dahmani2017hyperbolically}) and transitivity of hyperbolic embeddings (\cite[Proposition~4.35]{dahmani2017hyperbolically}) we have that $\ip{g_n}$ is hyperbolically embedded in $G$. In particular, since $a$ has infinite order and is contained inside a hyperbolically embedded cyclic subgroup, \cite[Theorem~1.4]{osin2016acylindrically} says that there is an acylindrical action of $G$ on a hyperbolic space $X$ such that $a$ is loxodromic. It follows that $g_n$ is loxodromic for each $n$ too. By almost malnormality of hyperbolically embedded subgroups (\cite[Proposition~4.33]{dahmani2017hyperbolically}), \cite[Remark~6.2]{dahmani2017hyperbolically}, and \cite[Corollary~6.6]{dahmani2017hyperbolically}, it follows that $E(g_n) = \ip{g_n}$ for each $n$. 
    
    We note that for all $n\ge \frac{1}{2}C^2D$ and any $h\in B_S({|g_n|_S^{1/2}})\setminus\{e\}$ (the $|g_n|_S^{1/2}$-ball around the identity in the Cayley graph of $G$ with the generating set $S$), we have $h\not\in E(g_n)$. Indeed, $|g_n|_{\{a,b\}} = 2n+1 > C^2D$ which implies that either $h\not\in\bF_2$ or 
    $$|h|_{\{a,b\}} \leq C|h|_S \leq C|g_{{n}}|_S^{1/2} \leq C\sqrt{D}|g_{{n}}|^{1/2}_{\{a,b\}} < |g_n|_{\{a,b\}},$$
    while clearly every element $g'\in E(g_n)\setminus\{e\}$ is an element of $\bF_2$ with $|g'|_{\{a,b\}} \geq |g_n|_{\{a,b\}}$.

    Since all of the $g_n$ are conjugate, they all have equal stable length. Therefore we can apply Theorem \ref{thm:quantitative admissible path} to get a degree 6 polynomial $P$ such that whenever $h_1,\ldots,h_m\in B_S({|g_n|_S^{1/2}})$ and $|k_1|,\ldots,|k_m| \geq P(|g_n|_S)$ we have $h_1g_n^{k_1}\cdots h_mg_n^{k_m}\neq e$. 

    We define $\phi_N:G*\ip{z}\to G$ by $\phi_N|_G = \id$ and $\phi_N(z) = g_{2CN^2}^{P(D(4CN^2+1))}$. We can now conclude as in the proof of Proposition \ref{F2-is-selfless}. Clearly $\phi_N$ satisfies the subexponential growth condition, since $|g_{2CN^2}|_S \leq D|g_{2CN^2}|_{a,b} = D(4CN^2+1).$ It remains to show that $\phi_n$ is injective on $B_{S \cup\{z\}}(N)$. It suffices to show that
    $$h_1\ell_1\cdots h_{m}\ell_{m} \neq e$$
    for all $h_i \in G$ such that $|h_i|_{S} \le 2N$ for all $1 \le i \le m$ and $h_i\neq e$ for all $1 < i < m$, and for all $\ell_i = \phi_N(z)^{p_i}$ where the $p_i$ are nonzero integers. But this follows from the previous paragraph's assertion, since, taking $n = 2CN^2,$ we have $h_i\in B_S(2N) \subset B_S(\sqrt{(2n+1)/C}) \subset B_S(|g_n|_S^{1/2})$ while the powers of $g_n$ are larger, in magnitude, than $P(D(2n+1))$ and thus are also larger than $ P(|g_n|_S)$. (Here, we assume $P$ is increasing, which is true for sufficiently large values.)

\end{proof}

\subsection{Proving strict comparison for $C^*_r(G)$}

\begin{defn}
    Let $G\leq \Gamma$ be countable discrete groups. We say that $\Gamma$ is \emph{existentially $C^\ast$-residually-$G$} (terminology motivated by \cite{louder2022strongly}) if for all finite sets $S\subset \Gamma$ and $\ee>0,$ there is a homomorphism $\phi:\Gamma \to G$ such that $\phi|_G = \id_G$ and for all $z\in\bC\Gamma$ supported on $S$ with $\|z\|_1 = 1,$ we have $\|\lambda_G(\phi(z))\| \leq \|\lambda_\Gamma(z)\| + \ee.$
\end{defn}

\begin{thm}\label{rapiddecaytrick}
    Suppose $(G,X)$ is selfless and has rapid decay. Then $G * \ip{a} \cong G * \bZ$ is existentially $C^*$-residually-$G$.
\end{thm}

\begin{proof}
We follow the proof of \cite[Theorem~1.6]{louder2022strongly}. Set $Y = X\cup\{a\}$ to be our finite generating set of $\Gamma := G * \bZ$. Let $\ee>0$ and let $S\subset \Gamma$ be a finite subset. Then there is some $R>0$ such that $S\subset B_Y(R)$. Let $S_Y(R)\subset \bC\Gamma$ denote the elements supported on $B_Y(R)$ with $\ell^1$ norm equal to 1. Using compactness of $S_Y(R)$ with respect to the $\ell^1$ norm, we find a finite $\frac{\ee}{3}$-net $\{a_i\}_{i\in I}$ for $S_Y(R)$. That is, every element in $S_Y(R)$ would be within $\ee/3$ distance of one of the $a_i$’s. It now suffices to prove the existence of a homomorphism $\phi:\Gamma \to G$ such that $\phi|_G=\id_G$ and that, for all $i\in I,$

    $$\|\lambda_G(\phi(a_i))\| \leq \|\lambda_\Gamma(a_i)\| + \frac{\ee}{3}.$$

    Since $G$ is selfless, we get a function $f:\bN\to\bR_+$ such that $\liminf_n f(n)^{\frac{1}{n}} = 1$ and homomorphisms $\phi_n: \Gamma\to G$ which are epimorphisms, are identity on $G,$ are injective on $B_Y(n),$ and such that $\phi_n(B_Y(n)) \subset B_X(f(n))$. Using rapid decay of $G$ we get a polynomial $P$ such that $\left\|\lambda_G(a)\right\|\le P(r)\|a\|_2$ for all $a$ supported on $B_X(r)$. Now, as $\liminf_n f(n)^{\frac{1}{n}}=1$, we may pick $m$ such that $P(f(2mR))^{\frac{1}{2m}} < 1+ \frac{\ee}{3}$. Define $b_i = \phi_{2mR}(a_i)$ for each $i\in I$.

    Note that 
    $\|\lambda_\Gamma(a_i)\|^{2m} =\|\lambda_\Gamma(a_i^*a_i)^m\|$
    and similarly $\|\lambda_G(b_i)\|^{2m} =\|\lambda_G(b_i^*b_i)^m\|$. Since each $a_i$ is supported on $B_Y(R)$, $(a_i^*a_i)^m$ is supported on $B_Y(2mR)$, and hence each $(b_i^*b_i)^m = \phi_{2mR}((a_i^*a_i)^m) $ is supported on $B_X(f(2mR))$. Since $\phi_{2mR}$ is injective on $B_Y(2mR)$, we have that $\|(b_i^*b_i)^m\|_2 = \|(a_i^*a_i)^m\|_2$. Therefore

    \begin{align*}
        \|\lambda_G(b_i)\|^{2m} &= \|\lambda_G(b_i^*b_i)^m\| \\
        &\leq P(f(2mR))\|(b_i^*b_i)^m\|_2 \\
        &= P(f(2mR))\|(a_i^*a_i)^m\|_2 \\
        &\leq P(f(2mR))\|\lambda_\Gamma(a_i)\|^{2m}.
    \end{align*}
    Since we assumed that $\|a_i\|_1 = 1,$ we conclude that
    \begin{align*}
        \|\lambda_G(b_i)\| &\leq P(f(2mR))^{\frac{1}{2m}}\|\lambda_\Gamma(a_i)\|\\
        &\leq (1+\frac{\ee}{3})\|\lambda_\Gamma(a_i)\| \\
        &\leq \|\lambda_\Gamma(a_i)\| + \frac{\ee}{3}.
    \end{align*}
    
\end{proof}

\begin{rem}\label{simple} 
    In the above proof, the injectivity of $\phi_n$ on the radius $n$ ball $B_{X\cup\{a\}}(n)$ is used only to obtain the equality $\|(b_i^*b_i)^m\|_2 = \|(a_i^*a_i)^m\|_2$. If instead $\phi_n$ is $g(n)$-to-1 on the radius $n$ ball for some function $g:\bN\to\bN$ such that $\lim g(n)^{\frac{1}{n}} = 1,$ then by the Cauchy-Schwarz inequality we would still have $\|(b_i^*b_i)^m\|_2 \leq \sqrt{g(2mR)}\|(a_i^*a_i)^m\|_2$ and the rest of the proof would proceed as above. To summarize, in the definition of selfless, we can assert the weaker condition that $\phi_n$ is $g(n)$-to-1 for some subexponential function $g$. Note that in this setting, the group homomorphism $\phi=(\phi_n)_\cU : G*\bZ \to G^\cU$ is no longer necessarily injective, which isn't a requirement anyway to prove Proposition \ref{resid-G-implies-embedding}.  We suspect that this approach might also help resolve the question of strict comparison in broader generality.
\end{rem}

\begin{prop}\label{resid-G-implies-embedding}
    If $\Gamma$ is existentially $C^\ast$-residually-$G$, then for some free ultrafilter $\cU$ on $\bN$ there is a $\ast$-homomorphism $\phi:C^*_r(\Gamma) \to C^*_r(G)^\cU$ such that $\phi \circ \iota = \Delta$ where $\iota$ is the inclusion of $C_r^*(G)\subset C_r^*(\Gamma)$ arising from the group inclusion $G\le \Gamma$ and $\Delta$ is the diagonal embedding $C_r^*(G)\subset C_r^*(G)^\cU$.
\end{prop}

\begin{proof}
    Since $\Gamma$ is existentially $C^*$-residually-$G$, there are epimorphisms $\phi_n:\Gamma \to G$ which are identity on $G$ and such that for all $z\in\bC\Gamma$ with $\|z\|_1 = 1$ we have 
    $$\lim_{n\to\cU}\|\lambda_G(\phi_n(z))\| \leq \|\lambda_\Gamma(z)\|. $$
    In other words, for all $z\in\bC\Gamma$ with $\|z\|_1 = 1$ we have
    $$\|(\phi_n(z))_\cU\|_{C^*_r(G)^\cU} \leq \|z\|_{C^*_r(\Gamma)}.$$
    Since the map $\phi = (\phi_n)_\cU$ is a homomorphism, we immediately get the previous inequality for all $z\in \bC\Gamma$. Then the inequality implies that we can extend $\phi$ continuously to all of $C_r^*(\Gamma)$; thus, $\phi$ is the desired $*$-homomorphism.
\end{proof}

\begin{cor}\label{Final corollary}
    If $G * \bZ$ is existentially $C^*$-residually-$G$, then $C^*_r(G)$ is selfless in the sense of \cite{robert2023selfless}. In particular, $C^*_r(G)$ has strict comparison.
\end{cor}

\begin{proof}
    By \cite[Theorem~2.6]{robert2023selfless}, it suffices to show that, for some free ultrafilter $\cU$ on $\bN$ there is an injective $\ast$-homomorphism $\phi$ from the reduced free product $C^*_r(G) * C_r^*(\bZ)$ to $C^*_r(G)^\cU$ such that, composing with the canonical inclusion $C_r^*(G)\subset C_r^*(G) * C_r^*(\bZ)$, we obtain the diagonal embedding $C_r^*(G)\subset C_r^*(G)^\cU$. The content of Proposition \ref{resid-G-implies-embedding} is precisely producing such a $*$-homomorphism, and it is injective because $G*\bZ$ is $C^*$-simple for any infinite group $G$, (see for example Corollary 3 in \cite{DE_LA_HARPE_2011}), and as a consequence, any $*$-homomorphism with domain $C^*_r(G*\bZ)$ will be automatically injective.
\end{proof}

\begin{rem}\label{rem:subsemigroup}
    In \cite{DYKEMA1999591} and \cite{Osinacylindrical}, the existence of sufficiently many free sub-semigroups of $G$ with a form of the rapid decay property, which they call the $\ell^2$-spectral radius property, is sufficient to show stable rank 1 of $C^*_r(G).$ In our setting, to show that $C^*_r(G)$ is selfless and therefore has strict comparison, we need an embedding of $C^*_r(G * \bZ)$ into $C^*_r(G)^\cU$ which requires estimating the operator norm on all of $\bC(G*\bZ)$ simultaneously. This is precisely why removing the rapid decay assumption is apriori difficult. 
\end{rem}

\begin{proof}[Proof of Theorem \ref{main2}]
    If $G$ is as in the hypothesis of Theorem \ref{main2}, then by Theorem \ref{mainselfless} and Theorem \ref{rapiddecaytrick} $G*\mathbb{Z}$ is existentially $C^*$-
residually-$G$, so that by Corollary \ref{Final corollary} $C^*_r(G)$ has strict comparison, and is in fact selfless.
\end{proof}

\subsection{On the particular case of free groups}\label{mageesubsection} After initial versions of our paper were released, M. Magee pointed out to us another proof via \cite{louder2022strongly} in the case of non-abelian free groups (avoiding Proposition \ref{F2-is-selfless}). We thank him sincerely for allowing us to include this here. 

\begin{thm}[\cite{louder2022strongly}]
    Let $n>1$ and $a\in \mathbb{F}_n$ be a non identity element. Denote by $G= \mathbb{F}_n*_{\langle a\rangle}\langle a\rangle\times \mathbb{Z}$, also known as an \emph{extension of centralizers} of $a$. Then $G$ is existentially $C^*$-residually-$\mathbb{F}_n$. 
\end{thm}

\begin{proof}
    We consider the maps defined in the proof of Proposition 1.8 of \cite{louder2022strongly}, proved using their quantitative Baumslag lemma. Within the proof, note that the maps $f$ constructed are actually identity on the base free group $\mathbb{F}_n$ (they are moreover composition of an automorphism fixing $\mathbb{F}_n$ and natural retract to $\mathbb{F}_n$). Then by following Theorem 1.6 of \cite{louder2022strongly}, the desired conclusion is reached.
\end{proof}

Note that for the applications described in this paper (via applying Proposition \ref{resid-G-implies-embedding}), we need the degenerate case where $a=1_{\mathbb{F}_n}$ in the above (which is for instance handled directly by our Proposition \ref{F2-is-selfless}). While this is not handled directly in the above argument, it is possible to deduce this case via the following observation: take $n=2$, $\mathbb{F}_2=\langle x,y\rangle$ and $a=y$, in the above Theorem. Then clearly, $G\cong\mathbb{Z}*(\mathbb{Z}\times \mathbb{Z})$, and let the generator commuting with $y$ be denoted $t$. Then the group generated by $x, y, txt$ is isomorphic to $\mathbb{F}_3$ as a subgroup of $G$, and contains the original $\mathbb{F}_2$. This can therefore serve to replace $G$ in the conclusion of the Theorem as needed. An identical argument works for any $n>2$ as well.

\appendix

\section{Quantitative version of Wenyuan Yang's admissible path lemma} \label{appendix main}

The aim of this section is to prove a quantitative version of W. Yang's admissible path lemma \cite[Corollary~3.4]{YANG_2014}, which applies to our setting. For our purposes, we will take a fixed action on a hyperbolic space but vary the loxodromic element and thus the quasi-geodesic constants $\lambda$. So throughout, we track dependence on $\lambda$. The following quantitative version of the Morse lemma is present in \cite[Section 5.2]{ghys2013groupes}; see \cite[Theorem~1.1]{SHCHUR2013815} for an explicit polynomial.

\begin{lem}
\label{lem:quantmorse}[Quantitative Morse Lemma]
    Let $X$ be a $\delta$-hyperbolic space, and let $\gamma$ be a $(\lambda,0)$-quasi-geodesic. Let $\sigma$ be a geodesic with the same endpoints as $\gamma$. Then $\gamma$ is contained in a $Q_1(\lambda)$-neighborhood of $\sigma$ with $Q_1$ a quadratic polynomial (which depends on $\delta$).
\end{lem}

We continue with some notation. Let $(G,S)\curvearrowright X$ be a group $G$ with a finite generating set $S$ acting acylindrically on a $\delta$-hyperbolic space $X$. Fix an origin $o\in X$. Let $g\in G$ be a loxodromic element. Let $\tau(g) = \lim_n \frac{d(g^no,o)}{n}$ denote the stable length of $g$. We observe that $\tau(g)$ is independent of the choice of base point and is stable under conjugation. Let $[g]$ denote the displacement (or translation length) of $g$, defined by\[[g]=\inf_{x\in X}d(gx,x).\]
This infimum is in fact achieved at some point $x_0\in X$ (see the discussion after Definition 5 in \cite{delzant}). We let $L_g = \cup_{n\in\bZ}[g^nx_0,g^{n+1}x_0]$. Let $\gamma = \cup_{n\in\bZ}[g^no,g^{n+1}o]$ be the axis of $g$ going through $o.$ Let $\lambda = \frac{d(go,o)}{\tau(g)}$ so that $\gamma$ is a $(\lambda,0)$-quasi-axis; i.e.,  the length of any sub-segment of $\gamma$ is at most $\lambda$ times the distance between the endpoints. That is, if $m < n$ are integers then $$\sum_{i=m}^{n-1}d(g^io,g^{i+1}o) \leq = \lambda d(g^{n-m}o,o).$$ Note that clearly $[g] \leq d(go,o) = \tau(g)\lambda.$ Let $E(g)$ denote the maximal elementary (i.e., virtually cyclic) subgroup of $G$ containing $g.$ For sets $Y,Z\subset X$, let $\pi_Y(Z)$ denote the projection of $Z$ onto $Y.$ More precisely, $y\in \pi_Y(Z)$ if and only if $y\in Y$ and there exists $z\in Z$ such that $d(y,z) = \min\{d(y',z) : y'\in Y\}$. For a set $Y\subset X$ let $\|Y\|$ denote its diameter; i.e., the smallest real number such that $d(y,y')\leq \|Y\|$ for all $y,y'\in Y.$

The following is due to Delzant. 

\begin{prop}\label{lem:weakly acyl}
    (\cite[Proposition~5]{delzant}) Let $G\curvearrowright X$ be an acylindrical action on a hyperbolic space. Then there exists a constant $D_0$ such that for all loxodromic elements (i.e. hyperbolic isometries) $g\in G$ and for all $h\not\in E(g)$,
    $$\|\pi_{L_g}(hL_g)\|\leq D_0(1+[g]). $$
\end{prop}

The following lemma is the quantitative version of \cite[Lemma~2.23]{wan2024uniformexponentialgrowthgroups}.
\begin{lem}[Quantitative Bounded Projection]\label{lem:bdd-proj}
    Let $G\curvearrowright X$ be an acylindrical action on a hyperbolic space. Then there is a quadratic $Q_2$ such that whenever a loxodromic element $g$ and a point $x\in X$ gives a $(\lambda,0)$-quasi-axis $\gamma = \cup_{n\in\bZ}[g^nx,g^{n+1}x]$ and $h\not\in E(g),$ then 
    $$\|\pi_\gamma(h\gamma)\|\leq Q_2(\lambda)(1+[g]).$$
\end{lem}

\begin{proof}
    The quasi-axes $\gamma$ and $L_g$ both converge to the same points on the Gromov boundary of $X$. We may, therefore, apply the quantitative Morse lemma to transfer the inequality from Proposition \ref{lem:weakly acyl} to $\gamma$, up to a quadratic in $\lambda$.
\end{proof}

The following theorem, the main theorem of this section, is essentially a reproduction of \cite[Corollary~3.4]{YANG_2014} but where we have tracked all of the constants and dependencies to verify they grow polynomially.

\begin{thm}[Quantitative Admissible Path Theorem]\label{thm:quantitative admissible path}
    Let $(G,S)\curvearrowright X$ be a group $G$ with a finite generating set $S$ acting acylindrically on a $\delta$-hyperbolic space $X$. Fix a constant $C>0.$ There is a degree 6 polynomial $P$ (depending on $G,S,X,\delta,C$) such that if $g$ is loxodromic with $\tau(g) = C$, then for all group elements $h_1,\ldots,h_m\not\in E(g)$ such that $|h_i|_S \leq |g|_S$, the product $h_1g^{n_1}h_2g^{n_2}\cdots h_mg^{n_m} \neq e$ whenever $|n_i| \geq P(|g|_S)$ for all $1\leq i\leq m$ (with $n_m$ also being allowed to equal 0).
\end{thm}

\begin{proof}
     Set $\lambda = \frac{d(go,o)}{\tau(g)}$. Let $\cL$ be the collection of $(\lambda,0)$-quasi-geodesics in $X$; in particular, $\cL$ includes all geodesics since $\lambda \ge 1$. Let $\bX$ be defined by $\bX = \{\cup_{n=k}^{\ell}[hg^no,hg^{n+1}o] : h\not\in E(g)\}$ and note that $\bX\subset \cL.$ By Lemma \ref{lem:quantmorse} and the proof of \cite[Corollary~3.4]{bestvina2009characterization}, we have that $\bX$ is $(\mu,\ee)$-contracting with respect to $\cL$ where $\mu,\ee$ are uniformly bounded above by a quadratic $Q_3(\lambda)$. That is, whenever $q\in \cL$ and $p\in\bX$ such that $d(q,p) \geq Q_3(\lambda)$ then $\|\pi_p(q)\| < Q_3(\lambda)$. (In the notation of \cite{YANG_2014}, $\mu(\lambda',c'),\ee(\lambda',c') \le Q_3(\lambda)$ for all $\lambda',c' \in \bR$.) 

    We now obtain that for each $p\in\bX$, $p$ is $\sigma$-quasiconvex where $\sigma(U) = \frac{3}{2}Q_3(\lambda) + 2U$. Indeed, let $q$ be a geodesic in $X$ such that its endpoints $q^+$ and $q^-$ are in the $U$-neighborhood $N_U(p)$ of $p.$ Since $p$ is $(\mu,\ee)$-contracting, either $d(q,p) < \mu$ (and we are done) or $\|\pi_p(q)\| < \ee$. In particular, we have $d(q^+,q^-) \le \ee + 2U$. Since $q$ is a geodesic, for any $x\in q$ we have $d(q^+,x) + d(q^-,x) = d(q^+,q^-)$. 
    Therefore without loss of generality $d(q^+,x) \le \ee/2 + U$, implying that $d(x,\pi_p(x)) < 2U + \frac{3}{2}\ee$. Thus $q\subset N_{\sigma(U)}(p)$.

    By Lemma \ref{lem:bdd-proj}, if $p,p'\in\bX$ are quasi-geodesics corresponding to $h,h'$ in different left cosets of $E(g)$, then $p,p'$ have bounded projection $T$ where $T = Q_2(\lambda)(1+[g]) = O(\lambda^3)$ since $[g] \leq d(go,o) = \tau(g)\lambda$ (and the stable length $\tau(g)$ is fixed). We note that in \cite{YANG_2014}, the symbol $\tau$ is used for bounded projection but we have already used $\tau$ for stable length. By \cite[Lemma~2.7]{YANG_2014}, $p,p'$ also have $\nu$-bounded intersection where $\nu(U) = 2U + O(\lambda^3)$.
    \allowdisplaybreaks

    Now let $n_1,\ldots,n_k$ be integers and consider the path 
    \begin{align*}
        P &=[o,h_1o]\cup[h_1o,h_1go]\cup \ldots 
    \cup[h_1g^{n_1-1}o,h_1g^{n_1}o]\cup[h_1g^{n_1}o,h_1g^{n_1}h_2o]\cup\ldots \\
    &\ldots\cup [h_1g^{n_1}\cdots h_kg^{n_k-1}o,h_1g^{n_1}\cdots h_mg^{n_m}o];
    \end{align*} set $q_i$ to be the geodesic subsegment of $P$ where $h_i$ first appears and set $p_i$ to be the subpath of $P$ between $q_i$ and $q_{i+1}$ (with $p_m$ being the subpath of $P$ coming after $q_m$). More explicitly, for $i= 1,2$ we have that $q_i$ and $p_i$ are as follows:
    \begin{align*}
        q_1 &= [o,h_1o],\\
        q_2 &= [h_1g^{n_1}o,h_1g^{n_1}gh_2o],\\
        p_1 &= [h_1o,h_1g^{\sgn(n_1)}o]\cup\ldots\cup[h_1g^{n_1-\sgn(n_1)}o,h_1g^{n_1}o],\\
        p_2 &= [h_1g^{n_1}h_2o,h_1g^{n_1}h_2g^{\sgn(n_2)}o]\cup\ldots\cup[h_1g^{n_1}h_2g^{n_2-\sgn(n_2)}o,h_1g^{n_1}h_2g^{n_2}o]\\
    \end{align*} $P$ is a path from $o$ to $h_1g^{n_1}h_2g^{n_2}\cdots h_mg^{n_m}o $ in $X$. We note that each $p_i$ is a $(\lambda,0)$-quasi-geodesic in $\bX$ (and thus has two endpoints in some $p'\in\bX$), each $p_i$ has length at least $\tau(g)|n_i|,$ each $q_i$ has $C'|g|_S$-bounded projection to $p_i$ and $p_{i-1}$ (since $q_i$ is a geodesic of length at most $C'|h|_S\leq C'|g|_S$ where $C' = \max_{k\in S}d(ko,o)$; note that $C'$ is a constant dependent only on the group action and the choice of generating set), and any two $p_i,p_{i+1}$ have $\nu$-bounded intersection since each $h_i\not\in E(g)$ so $p_i,p_{i+1}$ lie in different quasi-axes of $g.$ We note that $\lambda = \tau(g)d(go,o) \leq \tau(g)C'|g|_S$ so that $\lambda = O(|g|_S)$.
 
    We are now ready to analyze the constants in the proof of \cite[Proposition~3.3]{YANG_2014}. We have $c=0$ everywhere.
    \begin{align*}
        C_{\lambda,0} &= \lambda (\mu(\lambda,0) + \epsilon(\lambda,0) + C'|g|_S)\\ 
        &= O(|g|_S^3) \\
        B_{\lambda,0} &= 2\ee(1,0) + 2\mu(1,0) + \nu(\mu(1,0) + \sigma(0)) + C'|g|_S \\
        &= O(|g|^2) \\
        R &= \max(C'|g|_S + 2\ee(1,0) + 4\mu(1,0) + 1,\\
        &\quad\quad\quad \ \  \mu(1,0) + 5\ee(1,0) + B_{\lambda,0} + 1)\\
        &= O(|g|_S^2)\\
        \Lambda &= \lambda(6R + 1) \\
        &= O(|g|_S^3)\\
        D &= \max(\mu(1,0) + \ee(1,0) + C'|g|_S + C_{\lambda,0},\\
        &\quad\quad\quad \ \ 2C'|g|_S + 3\ee(1,0) + 6\mu(1,0),\\
        &\quad\quad\quad \ \ \Lambda(R + \sigma(\mu(1,0))), \\
        &\quad\quad\quad \ \ 13\ee(1,0) + 6\mu(1,0) + 2B_{\lambda,0}) \\
        &= O(|g|_S^5).
    \end{align*}
    Now, if we choose all of the $n_i$ so that $|n_i|$ are larger than $\lambda D = O(|g|_S^6),$ we can apply \cite[Corollary~3.4]{YANG_2014} to obtain that our path $P$ is a $(\Lambda,0)$-quasi-geodesic. In particular, $P$ is nontrivial and so it must be that $h_1g^{n_1}h_2g^{n_2}\cdots h_mg^{n_m} \neq e$. (Note that in the case $n_m=0$, we can still apply \cite[Corollary~3.4]{YANG_2014}.) This finishes the proof.
\end{proof}

\bibliographystyle{plain}
\bibliography{comparison}
\end{document}